\documentclass[11pt]{article}
\usepackage{amsmath,amsxtra,amssymb,amsthm,amsfonts}
\usepackage{graphicx}
\usepackage[margin=1.05in]{geometry}
\usepackage[dvipsnames]{xcolor}
\usepackage{xurl}
\usepackage[pdftex,linktocpage=true,colorlinks,citecolor=blue,linkcolor=blue,pagebackref]{hyperref}
\usepackage{mathpazo}

\usepackage{caption}
\usepackage{pgfplots}
\usepackage{pgf}
\usepackage{tikz}
\usetikzlibrary{patterns}
\usetikzlibrary{arrows.meta}
\usepgfplotslibrary{patchplots}
\usetikzlibrary{pgfplots.patchplots}
\pgfplotsset{width=9cm,compat=1.5.1}

\newcommand\ip[2]{\ensuremath{\langle#1 , #2\rangle}} 

\newcommand\supp{\mathop{\rm supp}\nolimits}

\newcommand{\defeq}{\stackrel{\smash{\textnormal{\tiny def}}}{=}}

\def\N{\mathbb{N}}

\def\R{\mathbb{R}}

\def\M{\mathcal{M}}

\def\a{\mathbf{a}}

\def\e{\mathbf{e}}

\def\v{\mathbf{v}}
\def\w{\mathbf{w}}
\def\x{\mathbf{x}}
\def\y{\mathbf{y}}

\def\0{\mathbf{0}}

\numberwithin{equation}{section}
\newtheorem{lemma}{Lemma}[section]

\newtheorem{theorem}[lemma]{Theorem}
\newtheorem{corollary}[lemma]{Corollary}
\newtheorem{conjecture}[lemma]{Conjecture}

\newtheorem{example}{Example}

\begin{document}

\title{Hadamard powers of Gram matrices and \\ Generalized Cauchy--Schwarz inequalities}

\author{Nathaniel Johnston\textsuperscript{1} and Mohammad Sababheh\textsuperscript{2}}

\maketitle

\begin{abstract}
    The following inequality for vectors $\v,\w \in \R^n$ was considered in [Linear and Multilinear Algebra 74:1779--1798 (2026)]:\[\|\v^p\|\|\w^p\| - \ip{\v^p}{\w^p} \leq \|\v\|^p\|\w\|^p - \ip{\v}{\w}^p,\]
    where $\v^p$ is the vector that is obtained by raising each entry of $\v$ to the power $p$. It was shown that this inequality holds for all positive integers $p$, and it was conjectured that it holds for all real $p \geq 2$ when the entries of $\v$ and $\w$ are positive. A formula for the maximum amount by which the inequality's right-hand side can exceed its left-hand side for nonnegative unit vectors was also conjectured. We prove both of these conjectures, and we determine exactly when the inequality holds with equality. We also present a new generalization of this inequality to three or more vectors.\medskip

    \medskip
	
\end{abstract}

\addtocounter{footnote}{1}
\footnotetext[1]{Department of Mathematics \& Computer Science, Mount Allison University, Canada}
\addtocounter{footnote}{1}
\footnotetext[2]{Program of Mathematics, Abdullah Al Salem University, Kuwait}

\section{Introduction}\label{sec:intro}

The Cauchy--Schwarz inequality is one of the most fundamental inequalities in all of mathematics. In its most basic form, it states that $|\ip{\v}{\w}|\leq\|\v\|\|\w\|$ for all vectors $\v, \w \in \R^n$, where $\ip{\v}{\w}$ is the usual inner (dot) product and $\|\v\|$ is the usual Euclidean norm. This inequality has been generalized in many different ways, including to more exotic inner product spaces, and to unitarily invariant norms of matrices (see \cite{Ste04,BD95}, for example).

Our starting point is the inequality
\begin{align}\label{eq:cs_exp2_original}
    \|\v^2\|\|\w^2\| - \ip{\v^2}{\w^2} \leq \|\v\|^2\|\w\|^2 - \ip{\v}{\w}^2,
\end{align}
where $\v^2$ and $\w^2$ are the vectors obtained by squaring each entry in $\v$ and $\w$, respectively. This inequality was introduced in~\cite{MO_Lemma,JM19}, and it can be thought of as a refinement of the Cauchy--Schwarz inequality: while the latter says that both sides are nonnegative, the former also provides a nontrivial lower bound on the right-hand side.

In \cite{JPTV26}, it was asked what effect changing the exponent $2$ in Inequality~\eqref{eq:cs_exp2_original} has. In particular, it was shown that
\begin{align}\label{eq:intro_power_cs}
    \|\v^p\|\|\w^p\| - \ip{\v^p}{\w^p} \leq \|\v\|^p\|\w\|^p - \ip{\v}{\w}^p
\end{align}
for all positive integers $p$, where $\v^p$ denotes the vector obtained by raising each entry of $\v$ to the power $p$. It was furthermore conjectured that if the entries of $\v$ and $\w$ are all positive then Inequality~\eqref{eq:intro_power_cs} holds for all real (not just integer) $p \geq 2$ \cite[Conjecture~5.1]{JPTV26} (and it was noted that this threshold is sharp: the inequality fails for every $p \in (1,2)$).

The distinction between integer and non-integer powers is familiar in the study of Hadamard (entrywise) powers of positive semidefinite matrices. The Schur product theorem \cite[Theorem~7.5.3]{HJ13} ensures that positive integer Hadamard powers preserve positive semidefiniteness. For non-integer powers, however, the matrix size matters: it was shown in \cite{FH77} that the set of $k \times k$ positive semidefinite entrywise nonnegative matrices is closed under positive non-integer Hadamard powers larger than $k-2$, but not smaller. Numerous related properties were established in \cite{GKR15}, including a complete determination of which Hadamard powers are superadditive on the positive semidefinite entrywise nonnegative cone. This superadditivity theorem is the main tool behind our proof of \cite[Conjecture~5.1]{JPTV26}. In fact, this technique gives an inequality involving any number of vectors.

A second question raised in \cite{JPTV26} asked how large the difference between the right- and left-hand sides of Inequality~\eqref{eq:intro_power_cs} can be. Since rescaling the vectors also rescales the difference, this question was asked specifically for unit vectors, and numerics suggested that the largest possible difference (for entrywise nonnegative vectors) is
\begin{align*}
    1 - \left(\left\lfloor\frac{n+1}{2}\right\rfloor\left\lceil\frac{n+1}{2}\right\rceil\right)^{\frac{1-p}{2}}.
\end{align*}
We also prove this conjecture in the affirmative.

The paper is organized as follows. In Section~\ref{sec:prelims}, we introduce our notation, recall the conjectures from \cite{JPTV26}, and state the superadditivity theorem that we will use. In Section~\ref{sec:first_conjecture}, we prove Inequality~\eqref{eq:intro_power_cs} (i.e., the first conjecture) and its multi-vector generalization. In Section~\ref{sec:equality_etc}, we develop a lower bound on the difference between the right- and left-hand sides of Inequality~\eqref{eq:intro_power_cs}, and we use this lower bound to determine exactly when the inequality holds with equality. Finally, in Section~\ref{sec:second_conjecture} we develop an \emph{upper} bound on the difference between the right- and left-hand sides of Inequality~\eqref{eq:intro_power_cs}, and thus prove the second conjecture.

\section{Mathematical Preliminaries}\label{sec:prelims}

We now introduce our notation and the mathematical basics behind our results.

\subsection{Notation}\label{sec:notation}

We take the convention that $\N = \{1,2,3,\ldots\}$. For vectors $\v, \w \in \R^n$, $\|\v\|$ denotes the Euclidean norm (2-norm), $\ip{\v}{\w}$ denotes the Euclidean inner product (dot product), and $\supp(\v) = \{j : v_j \neq 0\}$. We use $\e_1$, $\e_2$, $\ldots$, $\e_n$ to denote the standard basis vectors of $\R^n$. We use $\M_n$ to denote the set of $n \times n$ real matrices and $\M_n^{+}$ for its subset of positive semidefinite (PSD) matrices, which are always symmetric for us. The identity matrix in $\M_n$ is denoted by $I$. The inequality $X \succeq Y$ means that $X - Y$ is positive semidefinite. For $i,j \in \{1,2,\ldots,n\}$, the $(i,j)$-entry of $X \in \M_n$ is denoted by $[X]_{i,j}$.

We denote the Hadamard (entrywise) product of real matrices $X$ and $Y$ of the same size by $X \odot Y$. For a positive real number $p$ and a real matrix $X$ with nonnegative entries, $X^{\odot p}$ is the $p$-th Hadamard power of $X$ (i.e. for all $i$ and $j$, $[X^{\odot p}]_{i,j} = [X]_{i,j}^p$). For a vector $\v \in \R_{\geq 0}^n$ we instead write $\v^p$ for its $p$-th Hadamard power, since (unlike with matrices) there is no chance of confusing the superscript with another form of exponentiation.

We denote the Gram matrix of a collection of vectors $\v_1, \v_2, \ldots, \v_k \in \R^n$ (i.e., the matrix whose $(i,j)$-entry equals $\ip{\v_i}{\v_j}$) by $G(\v_1, \v_2, \ldots, \v_k)$. We note that every Gram matrix is positive semidefinite, and conversely every positive semidefinite matrix is the Gram matrix of some set of vectors (these facts are standard; see \cite[Exercise~2.2.25]{JohALA} for example).

\subsection{The conjectures}\label{sec:conjectures}

We now recall the two conjectures from \cite{JPTV26} that we will resolve. The first conjecture was proved for all positive \emph{integer} values of $p$, so the remaining difficulty was its extension to real values of $p$.

\begin{conjecture}\label{conj:non_integer}
    Let $p \geq 2$ be a real number. For any pair of vectors $\v,\w \in \R_{> 0}^n$, we have
    \begin{align*}
        \| \v^p \|\| \w^p \| - \ip{\v^p}{\w^p} \leq \|\v\|^{p}\|\w\|^{p} - \ip{\v}{\w}^{p}.
    \end{align*}
\end{conjecture}

We will prove this conjecture (see the upcoming Corollary~\ref{cor:non_integer}), and even generalize it to a version that applies to $3$ or more vectors. Furthermore, we will prove it for all entrywise nonnegative (not just entrywise positive) vectors.

The second conjecture was not explicitly listed as such, but was a question that was raised shortly after the above conjecture in \cite{JPTV26}: how much can the right-hand side and the left-hand side in Conjecture~\ref{conj:non_integer} differ by? It was noted that numerics suggest the following:\footnote{Strictly speaking, this observation was only made in \cite{JPTV26} for unit vectors; the form given here comes from rescaling.}

\begin{conjecture}\label{conj:rhs_minus_lhs}
    Let $p \geq 2$ be a real number. For any pair of vectors $\v,\w \in \R_{\geq 0}^n$, we have
    \begin{align*}
        \big(\|\v\|^{p}\|\w\|^{p} - \ip{\v}{\w}^{p}\big) - \big(\| \v^p \|\| \w^p \| - \ip{\v^p}{\w^p}\big) \leq \left[1 - \left(\left\lfloor\frac{n+1}{2}\right\rfloor\left\lceil\frac{n+1}{2}\right\rceil\right)^{\frac{1-p}{2}}\right]\|\v\|^p\|\w\|^p.
    \end{align*}
\end{conjecture}

In other words, if we define
\begin{align*}
    \Delta_p(\v,\w) \defeq \big(\|\v\|^{p}\|\w\|^{p} - \ip{\v}{\w}^{p}\big) - \big(\| \v^p \|\| \w^p \| - \ip{\v^p}{\w^p}\big)
\end{align*}
then these two conjectures, taken together, say that
\[
    0 \leq \Delta_p(\v,\w) \leq \left[1 - \left(\left\lfloor\frac{n+1}{2}\right\rfloor\left\lceil\frac{n+1}{2}\right\rceil\right)^{\frac{1-p}{2}}\right]\|\v\|^p\|\w\|^p
\]
for all entrywise positive vectors $\v$ and $\w$. The strange constant on the right-hand side of this conjecture comes from noting that if $\v = (1,1,\ldots,1,0,0,\ldots,0)$ has $\lfloor (n+1)/2\rfloor$ nonzero entries and $\w = (0,0,\ldots,0,1,1,\ldots,1)$ has $\lceil (n+1)/2\rceil$ nonzero entries, then equality is attained in Conjecture~\ref{conj:rhs_minus_lhs}, and numerics were unable to find any pair of vectors with the same $\|\v\|$ and $\|\w\|$ for which the left-hand side was larger. We will show in Theorem~\ref{thm:rhs_minus_lhs} that this conjecture is indeed true.

\subsection{Superadditivity of Hadamard powers}\label{sec:hadamard_superadditive}

The primary tool that gets our proofs off the ground is the following result of Guillot, Khare, and Rajaratnam \cite[Theorem~5.1]{GKR15}, which says that Hadamard powers are superadditive on the entrywise nonnegative positive semidefinite matrices, as long as the exponent is a positive integer or at least as large as the matrix size:

\begin{theorem}\label{thm:hadamard_superadditivity}
	Suppose $k,n \in \N$ and $X_1, X_2, \ldots, X_n \in \M_k^{+}$ have nonnegative entries. For every $p \in \N \cup (k,\infty)$, we have
	\begin{align*}
	    \left( \sum_{j=1}^n X_j \right)^{\odot p} \succeq \sum_{j=1}^n X_j^{\odot p}.
	\end{align*}
\end{theorem}

The above theorem is perhaps more frequently stated in the following form: If $X,Y \in \M_k^{+}$ have nonnegative entries and $p \in \N \cup (k,\infty)$, then
\begin{align}\label{eq:hadamard_super2}
	(X+Y)^{\odot p} \succeq X^{\odot p} + Y^{\odot p}.
\end{align}
Indeed, this inequality is just the $n = 2$ case of Theorem~\ref{thm:hadamard_superadditivity}. Conversely, applying Inequality~\eqref{eq:hadamard_super2} a total of $n-1$ times recovers Theorem~\ref{thm:hadamard_superadditivity}. Furthermore, the set of possible values of the exponent $p$ that is stated by the theorem is optimal: for any non-integer $p \in (0,k)$, there exist $X,Y \in \M_k^{+}$ with nonnegative entries for which Inequality~\eqref{eq:hadamard_super2} fails.

\section{Conjecture~\ref{conj:non_integer} and a generalization to more vectors}\label{sec:first_conjecture}

We now present a generalization of Conjecture~\ref{conj:non_integer} to any number of vectors. First, given a positive real number $p$ and vectors $\v_1$, $\v_2$, $\ldots$, $\v_k \in \R_{\geq 0}^n$, we define
\begin{align*}
    D_p(\v_1,\v_2,\ldots,\v_k) \defeq G(\v_1,\v_2,\ldots,\v_k)^{\odot p} - G(\v_1^p,\v_2^p,\ldots,\v_k^p).
\end{align*}
We use superadditivity of Hadamard powers (Theorem~\ref{thm:hadamard_superadditivity}) to show that $D_p(\v_1,\v_2,\ldots,\v_k)$ is positive semidefinite:

\begin{theorem}\label{thm:multivector}
    Let $k,n \in \N$, $\v_1, \v_2, \ldots, \v_k \in \R_{\geq 0}^n$, and $p \in \N \cup (k,\infty)$. Then $D_p(\v_1,\v_2,\ldots,\v_k)$ is positive semidefinite. Equivalently,
    \begin{align}\label{eq:multivector}
        \left\|\sum_{j=1}^k a_j\v_j^p\right\|^2 \leq \sum_{i,j=1}^k a_i a_j\ip{\v_i}{\v_j}^p
    \end{align}
    for all $a_1, a_2, \ldots, a_k \in \R$.
\end{theorem}

\begin{proof}
    Let $V$ be the $n \times k$ matrix with $\v_1$, $\v_2$, $\ldots$, $\v_k$ as its columns, in that order. For each $r \in \{1,2,\ldots,n\}$, define $\x_r^T \in \R_{\geq 0}^k$ to be the $r$-th row of $V$, so that
    \begin{align*}
        G(\v_1,\v_2,\ldots,\v_k) & = V^TV = \sum_{r=1}^n\x_r\x_r^T \qquad \text{and} \\
        G\big(\v_1^p,\v_2^p,\ldots,\v_k^p\big) & = (V^{\odot p})^T(V^{\odot p}) = \sum_{r=1}^n\big(\x_r\x_r^T\big)^{\odot p}.
    \end{align*}
    Each $\x_r\x_r^T \in \M_k^{+}$ has nonnegative entries, so Theorem~\ref{thm:hadamard_superadditivity} tells us that
    \begin{align*}
        G(\v_1,\v_2,\ldots,\v_k)^{\odot p} = \left(\sum_{r=1}^n\x_r\x_r^T\right)^{\odot p} \succeq \sum_{r=1}^n\big(\x_r\x_r^T\big)^{\odot p} = G\big(\v_1^p,\v_2^p,\ldots,\v_k^p\big).
    \end{align*}
    This means exactly that $D_p(\v_1,\v_2,\ldots,\v_k)$ is positive semidefinite, as desired.
    
    To prove Inequality~\eqref{eq:multivector}, note that positive semidefiniteness of $D_p(\v_1,\v_2,\ldots,\v_k)$ means exactly that, for all $\a := (a_1, a_2, \ldots, a_k)^T \in \R^k$, we have
    \begin{align*}
        0 \leq \a^T D_p(\v_1,\v_2,\ldots,\v_k)\a & = \a^T G(\v_1,\v_2,\ldots,\v_k)^{\odot p}\a - \a^T G(\v_1^p,\v_2^p,\ldots,\v_k^p)\a \\
        & = \sum_{i,j=1}^k a_i a_j\ip{\v_i}{\v_j}^p - \sum_{i,j=1}^k a_i a_j\ip{\v_i^p}{\v_j^p} \\
        & = \sum_{i,j=1}^k a_i a_j\ip{\v_i}{\v_j}^p - \left\|\sum_{j=1}^k a_j\v_j^p\right\|^2.
    \end{align*}
    This completes the proof.
\end{proof}

The set of permissible exponents $\N \cup (k,\infty)$ in Theorem~\ref{thm:hadamard_superadditivity} is known to be sharp for all $n \geq 2$ even when the matrices $X_1$, $X_2$, $\ldots$, $X_n$ all have rank~$1$ (see \cite[Theorem~5.1(1)]{GKR15}), which immediately implies that the same set of permissible exponents in Theorem~\ref{thm:multivector} is sharp. However, for completeness we find it instructive to have an explicit example showing that the set $\N \cup (k,\infty)$ is sharp in Theorem~\ref{thm:multivector} for all $n \geq 2$:

\begin{example}\label{ex:exponent_range}
    Suppose $k \in \N$ and $p \in (0,k)$ is not an integer. Define $m := \lfloor p \rfloor$ and notice that $0 \leq m \leq k-1$. For $t > 0$ and $j \in \{1,2,\ldots,k\}$, define the vector
    \begin{align*}
        \v_j := \big(1,j\sqrt{t}\big)^T \in \R_{\geq 0}^2.
    \end{align*}
    Then the entries of $D_p(\v_1, \v_2, \ldots, \v_k)$ are
    \begin{align*}
        [D_p(\v_1, \v_2, \ldots, \v_k)]_{i,j} & = (1+tij)^p - 1 - t^p i^p j^p \\
        & = \left(\sum_{\ell=1}^m \binom{p}{\ell}t^\ell i^\ell j^\ell\right) - t^p i^p j^p + O(t^{m+1}) \quad \text{for all} \quad i,j \in \{1,2,\ldots,k\},
    \end{align*}
    where the second equality comes from the Taylor expansion of the function $t \mapsto (1+tij)^p$ around $t = 0$, and $\binom{p}{\ell} = p(p-1)\cdots(p-\ell+1)/\ell!$ (even though $p$ is not an integer). Define the vector $\a = (a_1, a_2, \ldots, a_k)^T$ by
    \begin{align*}
        a_j := \frac{(-1)^{m+1-j}}{j}\binom{m}{j-1}
    \end{align*}
    for $j \in \{1,2,\ldots,m+1\}$ and $a_j := 0$ for $j \in \{m+2,m+3,\ldots,k\}$. The formula for an $m$-th finite difference of the polynomial $f_\ell(x) := x^{\ell-1}$ (where $\ell \in \{1,2,\ldots,m\}$) at $x = 1$ gives
    \begin{align}\label{eq:mth_diff_a}
        \sum_{j=1}^k a_j j^\ell = \sum_{j=1}^k (-1)^{m+1-j}\binom{m}{j-1} f_\ell(j) = 0,
    \end{align}
    since every polynomial of degree less than $m$ has $m$-th finite difference equal to $0$.
    
    On the other hand, $\sum_{j=1}^k a_j j^p > 0$. Indeed, if $m = 0$ then this sum equals $1$. If $m \geq 1$ then this sum is the $m$-th finite difference of $g(x) := x^{p-1}$ at $x = 1$, which is positive by the repeated mean value theorem, since
    \begin{align}\label{eq:mth_diff_b}
        \frac{\mathrm{d}^m}{\mathrm{d}x^m}x^{p-1} = (p-1)(p-2)\cdots(p-m)x^{p-m-1} > 0.
    \end{align}
    Plugging Equations~\eqref{eq:mth_diff_a} and~\eqref{eq:mth_diff_b} into the formula for $[D_p(\v_1, \v_2, \ldots, \v_k)]_{i,j}$ then gives
    \begin{align*}
        \a^T D_p(\v_1,\ldots,\v_k)\a & = \sum_{\ell=1}^m \binom{p}{\ell}t^\ell\left(\sum_{j=1}^k a_j j^\ell\right)^2 - t^p\left(\sum_{j=1}^k a_j j^p\right)^2 + O(t^{m+1}) \\
        & = -t^p\left(\sum_{j=1}^k a_jj^p\right)^2 + O(t^{m+1}).
    \end{align*}
    This quantity is strictly negative for all sufficiently small $t > 0$, since $p < m+1$, so $D_p(\v_1, \v_2, \ldots, \v_k)$ is not positive semidefinite. Appending $0$ coordinates to each $\v_j$ gives similar examples in every dimension $n \geq 2$.
\end{example}

Conjecture~\ref{conj:non_integer} is a straightforward corollary of Theorem~\ref{thm:multivector}:

\begin{corollary}\label{cor:non_integer}
    Let $p \geq 2$ be a real number. For any pair of vectors $\v,\w \in \R_{\geq 0}^n$, we have
    \begin{align*}
        \| \v^p \|\| \w^p \| - \ip{\v^p}{\w^p} \leq \|\v\|^{p}\|\w\|^{p} - \ip{\v}{\w}^{p}.
    \end{align*}
\end{corollary}

\begin{proof}
    Note that if $\v = \0$ or $\w = \0$ then the inequality is trivial, so by rescaling we can assume without loss of generality that $\|\v\| = \|\w\| = 1$. Now take $k = 2$, $a_1 = 1$, $a_2 = -1$, $\v_1 = \v$, and $\v_2 = \w$ in Theorem~\ref{thm:multivector} to see that
    \begin{align*}
        \| \v^p \|^2 - 2\ip{\v^p}{\w^p} + \| \w^p \|^2 & = \big\| \v^p - \w^p \big\|^2 \\
        & \leq \ip{\v}{\v}^p - \ip{\v}{\w}^p - \ip{\w}{\v}^p + \ip{\w}{\w}^p \\
        & = 2 - 2\ip{\v}{\w}^p.
    \end{align*}
    Since $2\|\v^p\|\|\w^p\| \leq \|\v^p\|^2 + \|\w^p\|^2$ (e.g., by the AM--GM inequality), the above inequality implies
    \[
        2\|\v^p\|\|\w^p\| - 2\ip{\v^p}{\w^p} \leq 2 - 2\ip{\v}{\w}^p.
    \]
    Dividing both sides by $2$ proves the corollary for entrywise nonnegative unit vectors. Rescaling then proves it for arbitrary entrywise nonnegative vectors.
\end{proof}

\section{The equality case of, and an improvement to, Corollary~\ref{cor:non_integer}}\label{sec:equality_etc}

Corollary~\ref{cor:non_integer} established that $\Delta_p(\v,\w) \geq 0$. In this section, we improve upon this result by providing a nonnegative (and often strictly positive) lower bound on $\Delta_p(\v,\w)$. As a result of this stronger result, it will then be relatively straightforward to identify exactly when $\Delta_p(\v,\w) = 0$ occurs.

Since $\Delta_p$ is homogeneous in $\v$ and $\w$, it suffices to consider unit vectors. For $p > 2$, define
\begin{align*}
    C_p \defeq \frac{1}{8}\min\{p-1,2\}\min\{p-2,2\}.
\end{align*}
The right-hand side of the following theorem is nonnegative, and typically strictly positive, and thus provides a strengthening of Corollary~\ref{cor:non_integer} for all $p > 2$:

\begin{theorem}\label{thm:Deltap_lower_bound}
    Let $p > 2$ and let $\v, \w \in \R_{\geq 0}^n$ be unit vectors. Then
    \begin{align*}
        \Delta_p(\v,\w) \geq \frac{1}{2}\big(\|\v^p\|-\|\w^p\|\big)^2 + C_p\big(2 - \|\v^2\|^2 - \|\w^2\|^2\big)\big(1 - \ip{\v}{\w}^2\big).
    \end{align*}
\end{theorem}

\begin{proof}
    For $t \in [0,1]$, define
    \begin{align*}
        X_j & := \begin{bmatrix}
            v_j^2 & v_jw_j \\ v_jw_j & w_j^2
        \end{bmatrix}, \quad G := \sum_{i=1}^n X_i = \begin{bmatrix} 1 & \ip{\v}{\w} \\ \ip{\v}{\w} & 1 \end{bmatrix}, \quad \text{and} \quad H_j(t) := tG + (1-t)X_j
    \end{align*}
    for $j \in \{1,2,\ldots,n\}$. For $\x = (1,-1)^T$, define
    \begin{align*}
        f(t) := \x^T\left(\sum_{i=1}^n X_i \odot H_i(t)^{\odot(p-1)}\right)\x,
    \end{align*}
    and notice that
    \begin{align}\label{eq:deficit_decomposition}
        f(1) - f(0) = 2\Delta_p(\v,\w) - \big(\|\v^p\|-\|\w^p\|\big)^2.
    \end{align}
    In particular, this means that if we can bound the derivative of $f(t)$ then, by integrating the inequality, we will be able to bound $f(1) - f(0)$ and thus $\Delta_p(\v,\w)$. Well, differentiating $f$ gives
    \begin{align}\label{eq:fprime_xGH}
        f^\prime(t) = (p-1)\x^T\left(\sum_{i=1}^nX_i\odot(G-X_i) \odot H_i(t)^{\odot (p-2)}\right)\x.
    \end{align}
    Notice that $X_i$ and $G - X_i = \sum_{j \neq i}X_j$ are positive semidefinite for all $i$, so if we can find a lower bound on $H_i(t)^{\odot (p-2)}$ (in the positive semidefinite order) then we will get a lower bound on $f^\prime(t)$.

    To this end, first notice that the diagonal entries of $H_j(t)$ lie in the interval $[t,1]$. The matrix
    \[
        H_j(t) - t(1-\ip{\v}{\w})I = t\begin{bmatrix}
            \ip{\v}{\w} & \ip{\v}{\w} \\ \ip{\v}{\w} & \ip{\v}{\w}
        \end{bmatrix} + (1-t)X_j
    \]
    is PSD and entrywise nonnegative, so its $(p-2)$-nd Hadamard power is positive semidefinite by \cite[Theorem~2.2]{FH77}. Furthermore, $H_j(t)^{\odot(p-2)} - \big(H_j(t) - t(1-\ip{\v}{\w})I\big)^{\odot(p-2)}$ equals
    \begin{align*}
        \begin{bmatrix}
            \big(t + (1-t)v_j^2\big)^{p-2} - \big((1-t)v_j^2 + t\ip{\v}{\w}\big)^{p-2} & 0 \\
            0 & \big(t + (1-t)w_j^2\big)^{p-2} - \big((1-t)w_j^2 + t\ip{\v}{\w}\big)^{p-2}
        \end{bmatrix}
    \end{align*}
    so $H_j(t)^{\odot(p-2)} \succeq D$, where $D$ is this diagonal matrix.
    
    Notice that if $a \in [t,1]$ is a diagonal entry of $H_j(t)$ (i.e., $a = t + (1-t)v_j^2$ or $a = t + (1-t)w_j^2$), then
    \begin{align*}
        a^{p-2} - \big(a-t(1-\ip{\v}{\w})\big)^{p-2} \geq \ell(t) := \begin{cases}
            (p-2)t(1-\ip{\v}{\w}) & \text{if} \ \ 2 < p \leq 3,\\
            t^{p-2}(1-\ip{\v}{\w}^{p-2}) & \text{if} \ \ p \geq 3.
        \end{cases}
    \end{align*}
    Indeed, the first inequality follows from applying the mean value theorem to the function $x \mapsto x^{p-2}$: there exists $c \in [a - t(1-\ip{\v}{\w}), a]$ such that
    \[
        a^{p-2} - (a - t(1-\ip{\v}{\w}))^{p-2} = (p-2)t(1-\ip{\v}{\w})c^{p-3} \geq (p-2)t(1-\ip{\v}{\w}).
    \]
    The second inequality follows because $a^{p-2} - (a-t(1-\ip{\v}{\w}))^{p-2}$ is nondecreasing in $a \in [t,1]$ when $p \geq 3$, so its minimum is attained when $a = t$.
    
    It follows that $H_j(t)^{\odot(p-2)} \succeq \ell(t)I$ for all $j$. Plugging this lower bound into Equation~\eqref{eq:fprime_xGH} and using the Schur product theorem \cite[Theorem~7.5.3]{HJ13} shows that
    \begin{align*}
        f^\prime(t) & \geq (p-1)\x^T\left(\sum_{i=1}^nX_i\odot(G-X_i) \odot \ell(t)I\right)\x \\
        & = (p-1)\sum_{i=1}^n\left(v_i^2(1-v_i^2)+w_i^2(1-w_i^2)\right)\ell(t) \\
        & = (p-1)\big(2 - \|\v^2\|^2 - \|\w^2\|^2\big)\ell(t).
    \end{align*}
    Integrating this inequality from $t = 0$ to $t = 1$ gives
    \begin{align}\label{ineq:f1_f0}
        f(1) - f(0) \geq \big(2 - \|\v^2\|^2 - \|\w^2\|^2\big)\begin{cases}
            \dfrac{(p-1)(p-2)}{2}\big(1-\ip{\v}{\w}\big) & \text{if} \ \ 2 < p \leq 3,\\[2mm]
            1-\ip{\v}{\w}^{p-2} & \text{if} \ \ p \geq 3.
        \end{cases}
    \end{align}

    We now note two additional inequalities: $1 - \ip{\v}{\w} \geq (1-\ip{\v}{\w}^2)/2$ and $1 - \ip{\v}{\w}^{p-2} \geq \min\{(p-2)/2,1\}(1-\ip{\v}{\w}^2)$. The former inequality is straightforward, while the latter follows from concavity of the function $x \mapsto x^{(p-2)/2}$ when $2 < p \leq 4$ and from $\ip{\v}{\w}^{p-2} \leq \ip{\v}{\w}^2$ when $p \geq 4$. Plugging these two inequalities, as well as Equation~\eqref{eq:deficit_decomposition}, into Inequality~\eqref{ineq:f1_f0} gives
    \[
        2\Delta_p(\v,\w) \geq \big(\|\v^p\|-\|\w^p\|\big)^2 + 2C_p\big(2 - \|\v^2\|^2 - \|\w^2\|^2\big)\big(1 - \ip{\v}{\w}^2\big).
    \]
    Dividing both sides by $2$ proves the theorem.
\end{proof}

The next result completes the equality classification in Corollary~\ref{cor:non_integer}, including the additional possibilities at $p=2$.

\begin{theorem}\label{thm:equality_cases}
    Let $\v, \w \in \R_{\geq 0}^n$.
    \begin{itemize}
        \item If $p > 2$ then $\Delta_p(\v,\w) = 0$ if and only if $\v$ and $\w$ are linearly dependent or $\v$ and $\w$ each have at most one nonzero entry.

        \item If $p = 2$ then $\Delta_p(\v,\w) = 0$ if and only if $\|\w\|^2v_iv_j = \|\v\|^2w_iw_j$ for all $i \neq j$.
    \end{itemize}
\end{theorem}

\begin{proof}
    The ``if'' statements of the theorem are straightforward, so we only prove the ``only if'' statements. Furthermore, the theorem is trivial if $\v = \0$ or $\w = \0$, so (by homogeneity of $\Delta_p(\v,\w)$) we can assume without loss of generality that $\|\v\| = \|\w\| = 1$.

    When $p > 2$, Theorem~\ref{thm:Deltap_lower_bound}, together with $C_p > 0$ and $\|\v^2\|^2, \|\w^2\|^2 \leq 1$, shows that $\Delta_p(\v,\w) = 0$ forces either $\ip{\v}{\w} = 1$ or $\|\v^2\|^2 = \|\w^2\|^2 = 1$. These conditions mean, respectively, that $\v = \w$ or that both vectors are standard basis vectors, which gives the first half of the theorem.

    When $p = 2$, direct calculation shows that
    \begin{align*}
        \Delta_2(\v,\w) = \frac{1}{2}\big(\|\v^2\| - \|\w^2\|\big)^2 + \frac{1}{2}\sum_{\substack{i,j=1 \\ i\neq j}}^n(v_iv_j - w_iw_j)^2.
    \end{align*}
    It follows that $\Delta_2(\v,\w) = 0$ forces $v_iv_j = w_iw_j$ for all $i \neq j$, which gives the second half of the theorem.
\end{proof}

In particular, Theorem~\ref{thm:equality_cases} shows that, for $p > 2$, strictly positive vectors $\v$ and $\w$ have $\Delta_p(\v,\w) = 0$ if and only if they are scalar multiples of each other. The following corollary provides a quantitative version of this observation: if $\v$ and $\w$ are strictly positive unit vectors then it bounds how small $\Delta_p(\v,\w)$ can be in terms of $\|\v-\w\|$, which approaches $0$ as $\v$ and $\w$ get closer to being equal to each other.

\begin{corollary}\label{cor:interior_bound}
    Let $n \geq 2$, $p > 2$, and $0 < \varepsilon \leq 1/\sqrt{n}$. If $\v, \w \in \R_{\geq 0}^n$ are unit vectors that satisfy $v_j \geq \varepsilon$ and $w_j \geq \varepsilon$ for all $j \in \{1,2,\ldots,n\}$, then
    \begin{align*}
        C_p(n-1)\varepsilon^2\|\v-\w\|^2 \leq \Delta_p(\v,\w) \leq \frac{p}{2}\|\v-\w\|^2.
    \end{align*}
\end{corollary}

\begin{proof}
    Since $\max_j\{v_j^2\} \leq 1-(n-1)\varepsilon^2$, we have
    \begin{align*}
        \big\|\v^2\big\|^2 = \sum_{i=1}^n v_i^4 \leq \left(\sum_{i=1}^n v_i^2\right)\max_j\big\{v_j^2\big\} = \max_j\{v_j^2\} \leq 1-(n-1)\varepsilon^2,
    \end{align*}
    and similarly for $\w$. Theorem~\ref{thm:Deltap_lower_bound} then gives
    \begin{align*}
        \Delta_p(\v,\w) & \geq \frac{1}{2}\big(\|\v^p\|-\|\w^p\|\big)^2 + C_p\big(2 - \|\v^2\|^2 - \|\w^2\|^2\big)\big(1 - \ip{\v}{\w}^2\big) \\
        & \geq C_p\big(2 - \|\v^2\|^2 - \|\w^2\|^2\big)\big(1 - \ip{\v}{\w}^2\big) \\
        & \geq 2C_p(n-1)\varepsilon^2\big(1 - \ip{\v}{\w}^2\big) \\
        & \geq C_p(n-1)\varepsilon^2\|\v-\w\|^2,
    \end{align*}
    which is the left inequality.
    
    For the right inequality, the Cauchy--Schwarz inequality tells us that
    \begin{align*}
        \Delta_p(\v,\w) = 1 - \ip{\v}{\w}^p - \big( \|\v^p\|\|\w^p\| - \ip{\v^p}{\w^p} \big) \leq 1 - \ip{\v}{\w}^p.
    \end{align*}
    Since the function $x \mapsto x^p$ is convex, we have $x^p \geq 1 + p(x-1)$ for all $x \in [0,1]$. In particular, for $x = \ip{\v}{\w}$ this gives
    \begin{align*}
        \Delta_p(\v,\w) \leq 1 - \ip{\v}{\w}^p \leq p\big(1-\ip{\v}{\w}\big) = \frac{p}{2}\|\v-\w\|^2,
    \end{align*}
    which completes the proof.
\end{proof}

For $p = 2$, if $n \geq 3$ then it is similarly true that strictly positive vectors $\v$ and $\w$ have $\Delta_2(\v,\w) = 0$ if and only if they are scalar multiples of each other. However, in dimension $n = 2$ their coordinates can also be interchanged to achieve $\Delta_2(\v,\w) = 0$.

\section{Conjecture~\ref{conj:rhs_minus_lhs}}\label{sec:second_conjecture}

This section is devoted to proving Conjecture~\ref{conj:rhs_minus_lhs}. That is, our goal is to show that
\[
    \Delta_p(\v,\w) \leq \left[1 - \left(\left\lfloor\frac{n+1}{2}\right\rfloor\left\lceil\frac{n+1}{2}\right\rceil\right)^{\frac{1-p}{2}}\right]\|\v\|^p\|\w\|^p.
\]
To prove this inequality, we will first need a pair of helper lemmas.

\begin{lemma}\label{lem:scalar}
    Let  $q \geq 1$ and $c, x \geq 0$. Let $f(x) \geq g(x) \geq 0$ be the unique numbers satisfying
    \begin{align}\label{eq:fg}
        f(x)^q - g(x)^q = c^q \qquad \text{and} \qquad f(x)g(x) = x^2.
    \end{align}
    Then
    \begin{align}\label{eq:fg_ratios}
        \frac{f(x) - c}{x} \qquad \text{and} \qquad \frac{g(x)}{x}
    \end{align}
    are nondecreasing functions of $x$ on $(0,\infty)$, and they take values in $[0,1]$.
\end{lemma}

\begin{proof}
    If $c = 0$ then $f(x) = g(x) = x$, so both quotients equal $1$ (and thus the statement of the lemma holds). If $c > 0$ then, since $q \geq 1$, we have
    \begin{align*}
        f(x) = \big(c^q + g(x)^q\big)^{1/q} \leq c + g(x).
    \end{align*}
    It follows that
    \begin{align*}
        0\leq f(x)-c\leq g(x)\leq\sqrt{f(x)g(x)}=x.
    \end{align*}
    Dividing both sides by $x > 0$ proves that the ratios~\eqref{eq:fg_ratios} take values in $[0,1]$.

    To see that these ratios are nondecreasing in $x$, we compute their derivatives with respect to $x$:
    \[
        \frac{xf^\prime(x) - \big(f(x)-c\big)}{x^2} \qquad \text{and} \qquad \frac{xg^\prime(x) - g(x)}{x^2}.
    \]
    It thus suffices to prove that $xf^\prime(x) - \big(f(x)-c\big) \geq 0$ and $xg^\prime(x) - g(x) \geq 0$. To this end, differentiate both sides of Equations~\eqref{eq:fg} with respect to $x$ to get
    \[
        f(x)^{q-1}f^\prime(x) = g(x)^{q-1}g^\prime(x) \qquad \text{and} \qquad f^\prime(x)g(x) + f(x)g^\prime(x) = 2x.
    \]
    Multiplying both sides of the equation on the right by $g(x)^{q-1}$ and then substituting the left equation gives
    \[
        2xg(x)^{q-1} = f^\prime(x)g(x)^q + f(x)g(x)^{q-1}g^\prime(x) = f^\prime(x)(g(x)^q + f(x)^q).
    \]
    Multiplying both sides by $x$, dividing both sides by $f(x)^q + g(x)^q$ and then using the fact that $f(x)g(x) = x^2$ shows
    \begin{align}\label{eq:xfprime}
        xf^\prime(x) = \frac{2x^2g(x)^{q-1}}{f(x)^q + g(x)^q} = \frac{2f(x)g(x)^q}{f(x)^q + g(x)^q}.
    \end{align}
    An almost identical argument shows that
    \begin{align}\label{eq:xgprime}
        xg^\prime(x) = \frac{2g(x)f(x)^q}{f(x)^q + g(x)^q}.
    \end{align}
    By using Equation~\eqref{eq:xfprime} and plugging in $f(x)^q - g(x)^q = c^q$, we see that
    \begin{align*}
        xf^\prime(x) - \big(f(x)-c\big) & = c + \frac{f(x)(g(x)^q - f(x)^q)}{f(x)^q + g(x)^q} = c - \frac{f(x)c^q}{f(x)^q+g(x)^q} \geq 0,
    \end{align*}
    where the inequality follows from $f(x) \geq c$ and hence $f(x)c^{q-1}\leq f(x)^q$. Similarly, using Equation~\eqref{eq:xgprime} and plugging in $f(x)^q - g(x)^q = c^q$ shows that
    \begin{align*}
        xg^\prime(x) - g(x) & = \frac{g(x)(f(x)^q - g(x)^q)}{f(x)^q + g(x)^q} = \frac{g(x)c^q}{f(x)^q+g(x)^q} \geq 0.
    \end{align*}
    These are exactly the inequalities that we needed to prove, so we are done.
\end{proof}

The next lemma is motivated by the fact that the extremal pairs of vectors in Conjecture~\ref{conj:rhs_minus_lhs} (described in Section~\ref{sec:conjectures}) seem to be $\v = (1,1,\ldots,1,0,0,\ldots,0)$ with $\lfloor (n+1)/2\rfloor$ nonzero entries and $\w = (0,0,\ldots,0,1,1,\ldots,1)$ with $\lceil (n+1)/2\rceil$ nonzero entries, which have supports that overlap in a single element. This lemma will be used to show that this single-entry overlap of supports is indeed optimal.

\begin{lemma}\label{lem:concentration}
    Suppose $p \geq 2$ and $\x,\y \in \R_{\geq 0}^n$. Then there exist $\widetilde{\x}, \widetilde{\y} \in \R_{\geq 0}^n$ such that all of the following hold:
    \begin{itemize}
        \item $\widetilde{x}_j^p - \widetilde{y}_j^p = x_j^p - y_j^p$ for all $j \in \{1,2,\ldots,n\}$.

        \item $\ip{\widetilde{\x}}{\widetilde{\y}} = \ip{\x}{\y}$.

        \item $\|\widetilde{\x}\|\|\widetilde{\y}\| \geq \|\x\|\|\y\|$.

        \item $|\supp(\widetilde{\x}) \cap \supp(\widetilde{\y})| \leq 1$.
    \end{itemize}
\end{lemma}

\begin{proof}
    Define $h_j := x_jy_j$ for $j \in \{1,2,\ldots,n\}$ and $H := \sum_{j=1}^n h_j = \ip{\x}{\y}$. If $H = 0$ then $|\supp(\x) \cap \supp(\y)| = 0$, so assume from now on that $H > 0$. For each $j$, define
    \begin{align*}
        a_j := \max\big\{x_j^p - y_j^p, 0\big\}^{2/p} \qquad \text{and} \qquad b_j := \max\big\{y_j^p - x_j^p, 0\big\}^{2/p},
    \end{align*}
    and notice that $a_jb_j = 0$.
    
    Suppose for now that $x_j \geq y_j$. Applying Lemma~\ref{lem:scalar} when $q = p/2$, $c = a_j + b_j = |x_j^p - y_j^p|^{2/p}$, and $x = h_j$ gives $f(h_j) = x_j^2$ and $g(h_j) = y_j^2$ (or $f(h_j) = y_j^2$ and $g(h_j) = x_j^2$ if $x_j < y_j$). The bounds from that lemma then tell us that
    \begin{align}\label{ineq:fah}
        0 \leq f(h_j) - a_j = x_j^2 - a_j \leq h_j \quad \text{and} \quad 0 \leq g(h_j) = y_j^2 - b_j \leq h_j,
    \end{align}
    where the equality on the right uses the fact that $b_j = 0$. Now let $u_j, v_j \geq 0$ be the unique numbers satisfying
    \begin{align*}
        u_j^{p/2} - v_j^{p/2} = x_j^p - y_j^p \qquad \text{and} \qquad u_jv_j = H^2.
    \end{align*}
    In other words, $u_j$ and $v_j$ are the numbers $f(x)$ and $g(x)$ from Lemma~\ref{lem:scalar} when $q = p/2$, $c = a_j+b_j$, and $x = H$. Since $h_j \leq H$, applying monotonicity from that lemma shows that
    \begin{align}\label{ineq:ratios_monotonic}
        x_j^2 - a_j \leq \frac{h_j}{H}(u_j - a_j) \qquad \text{and} \qquad y_j^2-b_j \leq \frac{h_j}{H}(v_j-b_j).
    \end{align}
    If we instead have $x_j < y_j$ then interchanging the roles of $f(h_j)$ and $g(h_j)$, and the roles of $f(H)$ and $g(H)$, gives the same bounds on $x_j^2 - a_j$ and $y_j^2 - b_j$.

    For $i \neq j$, Inequality~\eqref{ineq:fah} gives
    \begin{align}\label{ineq:unequal_indices}
        (x_i^2 - a_i)(y_j^2 - b_j) \leq h_i h_j,
    \end{align}
    while $a_jb_j = 0$ gives
    \begin{align}\label{eq:equal_indices}
        (x_j^2 - a_j)(y_j^2 - b_j) = h_j^2 - a_j(y_j^2-b_j) - b_j(x_j^2-a_j).
    \end{align}
    If we define $A := \sum_{j=1}^n a_j$ and $B := \sum_{j=1}^n b_j$, then
    \begin{align}\begin{split}\label{eq:xy_sq_norm_expand}
        \|\x\|^2\|\y\|^2 & = \left(A + \sum_{j=1}^n\big(x_j^2 - a_j\big)\right)\left(B + \sum_{j=1}^n\big(y_j^2 - b_j\big)\right) \\
        & = AB + A\sum_{j=1}^n\big(y_j^2 - b_j\big) + B\sum_{j=1}^n\big(x_j^2 - a_j\big) + \sum_{i,j=1}^n\big(x_i^2 - a_i\big)\big(y_j^2 - b_j\big).
    \end{split}\end{align}
    Using Inequality~\eqref{ineq:unequal_indices} and Equation~\eqref{eq:equal_indices} shows that
    \begin{align*}
        \sum_{i,j=1}^n\big(x_i^2 - a_i\big)\big(y_j^2 - b_j\big) & \leq \sum_{\substack{i,j=1 \\ i\neq j}}^n h_i h_j + \sum_{j=1}^n \big( h_j^2 - a_j(y_j^2-b_j) - b_j(x_j^2-a_j) \big) \\
        & = H^2 - \sum_{j=1}^n \big( a_j(y_j^2-b_j) + b_j(x_j^2-a_j) \big).
    \end{align*}
    Plugging this into Equation~\eqref{eq:xy_sq_norm_expand} shows that
    \begin{align*}
        \|\x\|^2\|\y\|^2 & \leq H^2 + AB + A\sum_{j=1}^n\big(y_j^2 - b_j\big) + B\sum_{j=1}^n\big(x_j^2 - a_j\big) - \sum_{j=1}^n \big( a_j(y_j^2-b_j) + b_j(x_j^2-a_j) \big) \\
        & = H^2 + AB + \sum_{j=1}^n\big((A-a_j)(y_j^2-b_j) + (B-b_j)(x_j^2-a_j)\big) \\
        & \leq \sum_{j=1}^n\frac{h_j}{H}\big(H^2 + AB + (A-a_j)(v_j-b_j) + (B-b_j)(u_j-a_j)\big) \\
        & = \sum_{j=1}^n\frac{h_j}{H}(A-a_j+u_j)(B-b_j+v_j),
    \end{align*}
    where the second inequality comes from Inequality~\eqref{ineq:ratios_monotonic} and the fact that $\sum_{j=1}^n h_j/H = 1$, and the final equality uses the facts that $H^2 = u_jv_j$ and $a_jb_j = 0$.

    Since $h_j/H \geq 0$ and $\sum_{j=1}^n h_j/H = 1$, there exists some index $j$ for which
    \begin{align*}
        \|\x\|^2\|\y\|^2 \leq (A-a_j+u_j)(B-b_j+v_j).
    \end{align*}
    Choose this index and define
    \begin{align*}
        \widetilde{x}_i := \begin{cases}
            \sqrt{u_j} & \text{if } i = j,\\
            \sqrt{a_i} & \text{if } i \neq j.
        \end{cases} \qquad \text{and} \qquad
        \widetilde{y}_i := \begin{cases}
            \sqrt{v_j} & \text{if } i = j,\\
            \sqrt{b_i} & \text{if } i \neq j.
        \end{cases}
    \end{align*}
    Then $\supp(\widetilde{\x}) \cap \supp(\widetilde{\y}) = \{j\}$, $\ip{\widetilde{\x}}{\widetilde{\y}} = H = \ip{\x}{\y}$, and $\widetilde{x}_i^p - \widetilde{y}_i^p = x_i^p - y_i^p$ for all $i$. Furthermore,
    \begin{align*}
        \big\|\widetilde{\x}\big\|^2\big\|\widetilde{\y}\big\|^2 = (A-a_j+u_j)(B-b_j+v_j) \geq \|\x\|^2\|\y\|^2,
    \end{align*}
    which completes the proof.
\end{proof}

With these lemmas taken care of, we can now prove Conjecture~\ref{conj:rhs_minus_lhs}.

\begin{theorem}\label{thm:rhs_minus_lhs}
    Let $p \geq 2$ be a real number. For any pair of vectors $\v,\w \in \R_{\geq 0}^n$, we have
    \begin{align*}
        \Delta_p(\v,\w) \leq \left[1 - \left(\left\lfloor\frac{n+1}{2}\right\rfloor\left\lceil\frac{n+1}{2}\right\rceil\right)^{\frac{1-p}{2}}\right]\|\v\|^p\|\w\|^p.
    \end{align*}
\end{theorem}

\begin{proof}
    If either $\v$ or $\w$ equals $\0$ then the theorem is trivial, so we assume from now on that $\v, \w \neq \0$. Since both sides of the inequality are homogeneous of degree $p$ in each vector, we may assume without loss of generality that $\|\v\| = \|\w\| = 1$. With this assumption, our goal is to prove that
    \begin{align}\label{eq:deficit_complement}
        \|\v^p\|\|\w^p\| - \ip{\v^p}{\w^p} + \ip{\v}{\w}^p \geq \left(\left\lfloor\frac{n+1}{2}\right\rfloor\left\lceil\frac{n+1}{2}\right\rceil\right)^{\frac{1-p}{2}}.
    \end{align}
    
    For any real $t > 0$, the left-hand side is unchanged when $\v$ and $\w$ are replaced by $\x := t\v$ and $\y := \w/t$, respectively. Choose $t$ so that $\|\x^p\| = \|\y^p\|$. It is not the case in general that $\|\x\| = \|\y\| = 1$, but we do have $\|\x\|\|\y\| = 1$. The left-hand side of Inequality~\eqref{eq:deficit_complement} now equals
    \begin{align}\label{eq:preserved_quantity}
        N(\x,\y) := \frac{1}{2}\sum_{j=1}^n\big(x_j^p - y_j^p\big)^2 + \left(\sum_{j=1}^n x_j y_j\right)^p.
    \end{align}
    If we apply Lemma~\ref{lem:concentration} to $\x$ and $\y$ to obtain vectors $\widetilde{\x}$ and $\widetilde{\y}$, then we have $N(\widetilde{\x},\widetilde{\y}) = N(\x,\y)$ and $\|\widetilde{\x}\|\|\widetilde{\y}\| \geq \|\x\|\|\y\| = 1$. Furthermore, there exists at most one index $i$ for which $\widetilde{x}_i\widetilde{y}_i \neq 0$, so
    \begin{align*}
        \left(\sum_{j=1}^n\widetilde{x}_j\widetilde{y}_j\right)^p = \sum_{j=1}^n\widetilde{x}_j^p\widetilde{y}_j^p.
    \end{align*}
    Equation~\eqref{eq:preserved_quantity} thus simplifies to
    \begin{align*}
        N(\x,\y) = \frac{1}{2}\big(\|\widetilde{\x}^p\|^2 + \|\widetilde{\y}^p\|^2\big) \geq \|\widetilde{\x}^p\|\|\widetilde{\y}^p\|.
    \end{align*}
    Define $r := |\supp(\widetilde{\x})|$ and $s := |\supp(\widetilde{\y})|$. Then H\"older's inequality gives
    \begin{align*}
        \|\widetilde{\x}^p\| \geq r^{\frac{1-p}{2}}\|\widetilde{\x}\|^p \qquad \text{and} \qquad \|\widetilde{\y}^p\| \geq s^{\frac{1-p}{2}}\|\widetilde{\y}\|^p.
    \end{align*}
    Since $r + s \leq n+1$, it follows that
    \begin{align*}
        N(\x,\y) \geq (rs)^{\frac{1-p}{2}}\big(\|\widetilde{\x}\|\|\widetilde{\y}\|\big)^p \geq (rs)^{\frac{1-p}{2}} \geq \left(\left\lfloor\frac{n+1}{2}\right\rfloor\left\lceil\frac{n+1}{2}\right\rceil\right)^{\frac{1-p}{2}},
    \end{align*}
    which completes the proof.
\end{proof}

\section*{Acknowledgements}
The authors used ChatGPT-6 Astra for exploratory analysis, to simplify proofs, and to come up with some of the uglier inequalities used in the proofs of Lemmas~\ref{lem:scalar} and~\ref{lem:concentration}. The authors wrote the paper themselves and take full responsibility for its correctness. N.J.\ was supported by NSERC Discovery Grant number RGPIN-2022-04098.

\bibliographystyle{alpha}
\bibliography{ref}

\begin{thebibliography}{JPTV26}

\bibitem[BD95]{BD95}
Rajendra Bhatia and Chandler Davis.
\newblock A {C}auchy--{S}chwarz inequality for operators with applications.
\newblock {\em Linear Algebra and its Applications}, 223--224:119--129, 1995.

\bibitem[FH77]{FH77}
Carl~H. FitzGerald and Roger~A. Horn.
\newblock On fractional {H}adamard powers of positive definite matrices.
\newblock {\em Journal of Mathematical Analysis and Applications}, 61(3):633--642, 1977.

\bibitem[GKR15]{GKR15}
Dominique Guillot, Apoorva Khare, and Bala Rajaratnam.
\newblock Complete characterization of {Hadamard} powers preserving {Loewner} positivity, monotonicity, and convexity.
\newblock {\em Journal of Mathematical Analysis and Applications}, 425(1):489--507, 2015.

\bibitem[HJ13]{HJ13}
Roger~A. Horn and Charles~R. Johnson.
\newblock {\em Matrix Analysis}.
\newblock Cambridge University Press, second edition, 2013.

\bibitem[JM19]{JM19}
Nathaniel Johnston and Olivia MacLean.
\newblock Pairwise completely positive matrices and conjugate local diagonal unitary invariant quantum states.
\newblock {\em Electronic Journal of Linear Algebra}, 35:156--180, 2019.

\bibitem[Joh18]{MO_Lemma}
Nathaniel Johnston.
\newblock A strengthening of the {C}auchy--{S}chwarz inequality.
\newblock MathOverflow, 2018.
\newblock \url{https://mathoverflow.net/q/301844}.

\bibitem[Joh21]{JohALA}
Nathaniel Johnston.
\newblock {\em Advanced Linear and Matrix Algebra}.
\newblock Springer International Publishing, 2021.

\bibitem[JPTV26]{JPTV26}
Nathaniel Johnston, Sarah Plosker, Charles Torrance, and Luis M.~B. Varona.
\newblock Generalizing the {Cauchy--Schwarz} inequality: {Hadamard} powers and tensor products.
\newblock {\em Linear and Multilinear Algebra}, 74(13):1779--1798, 2026.

\bibitem[Ste04]{Ste04}
J.~Michael Steele.
\newblock {\em The {C}auchy--{S}chwarz master class: An introduction to the art of mathematical inequalities}.
\newblock Cambridge University Press, 2004.

\end{thebibliography}

\end{document}